\documentclass[11pt,leqno]{article}
\usepackage{latexsym,enumerate}
\usepackage{amsmath,amssymb}
\usepackage{amsthm}
\usepackage{float}
\usepackage{caption}
\title{The $D_{\pi}$-property on products of $\pi$-decomposable groups}
\author{L. S. Kazarin,
A. Mart\'{\i}nez-Pastor
 and M.~D. P\'{e}rez-Ramos}
\date{}
\newtheorem{lem}{Lemma}
\newtheorem{pro}{Proposition}
\newtheorem{teor}{Theorem}

\newtheorem*{mainth}{Main Theorem}

\theoremstyle{definition}

\newcommand{\Aut}{\text{Aut}}

\begin{document}
\maketitle

\begin{abstract}
The aim of this paper is to prove the following result:
Let $\pi$ be a set of odd primes. If the group $G=AB$ is the product of two
$\pi$-decomposable subgroups $A=A_\pi \times A_{\pi'}$ and $B=B_\pi \times B_{\pi'}$,
then $G$ has a unique conjugacy class of Hall $\pi$-subgroups, and any $\pi$-subgroup is contained in a Hall $\pi$-subgroup (i.e. $G$ satisfies property $D_{\pi}$).\\

\end{abstract}

\section{Introduction}

All groups considered in this paper are assumed to be finite. A well-known result in the framework of finite factorized groups is the classical theorem by O. Kegel and H. Wielandt which asserts the solubility of a group which is
the product of two nilpotent subgroups. This theorem has been the motivation for a great number of results in the literature on factorized groups. Particularly some of them consider the situation when either one or both factors are $\pi$-decomposable, for a set of primes
$\pi$.  A group $X$ is said to be $\pi$-\emph{decomposable} for a set of primes $\pi$,
 if $X=X_\pi \times X_{\pi'}$ is the direct product of a $\pi$-subgroup  $X_\pi$
 and a $\pi'$-subgroup  $X_{\pi'}$, where $\pi'$ stands for the complement of $\pi$ in the set of all
prime numbers. For any group $X$ and any set of primes $\sigma$, we use $X_\sigma$  to denote a Hall $\sigma$-subgroup of $X$.

In this line, Y. G.  Berkovich \cite{Ber} proved that the Kegel and Wielandt's result remains true for a group $G=AB$ which is the product of subgroups $A$ and $B$ such that one of the factors, say $A$, is $2$-decomposable, the other factor $B$ is nilpotent of odd order, and $A$ and $B$ have coprime orders. If the subgroup $B$ is metanilpotent instead of nilpotent, but preserving all the remaining conditions,
P. J. Rowley \cite{Row} proved that the group $G$ is $\sigma$-separable, for  the set $\sigma$  of all odd primes dividing the order of $A$. Z. Arad and D. Chillag \cite{A-C} showed that this conclusion  remains true without any restriction on the nilpotent length of $B$. Previously L. S. Kazarin \cite{Ka1} had obtained under the same hypotheses that $O_{2'}(A) \leq O_{2'}(G)$.

A significant extension of these results was obtained in \cite{paper1} by proving that $O_\pi(A)\le O_\pi (G)$ whenever $G=AB$, $A$ is a $\pi$-decomposable subgroup of $G$ for any set of odd primes $\pi$, and $B$ is a $\pi$-subgroup of $G$; equivalently, $O_\pi(A)B$ is a Hall $\pi$-subgroup of $G$ (see \cite[Theorem 1, Lemma 1]{paper1}). Under the additional hypothesis that $A$ and $B$ have coprime orders, such as considered in the mentioned previous results, it is easily
derived the $\sigma$-separability of $G$ for the set $\sigma$  of all odd prime divisors of the order of $A$.

In fact the results in \cite{paper1} make up the starting point of a longtime development carried out in \cite{paper1,paper2,bath,paper3,paper4}, where the existence of Hall $\pi$-subgroups have been considered as a preliminary step to  finding conditions of $\pi$-separability for products of $\pi$-decomposable subgroups. The goal was to prove the following theorem, which was first stated as a conjecture in \cite{paper2}:
\begin{teor}\textup{(\cite[Main Theorem]{paper4})}\label{mainth}
Let $\pi$ be a set of odd primes. Let the group $G=AB$ be the product of two
$\pi$-decomposable subgroups $A=A_\pi \times A_{\pi'}$ and $B=B_\pi \times B_{\pi'}$. Then
$A_\pi B_\pi=B_\pi A_\pi$ and this is a Hall $\pi$-subgroup of $G$. In particular, $G$ satisfies the property $E_{\pi}$.
\end{teor}

This theorem, whose proof uses deeply   the classification of finite simple groups (CFSG),  substantially extends all the above mentioned results. In particular, we have achieved some non-simplicity and $\pi$-separability criteria for certain products of $\pi$-decom\-posable groups, which contribute some new extensions  of the theorem of Kegel and Wielandt (see also \cite{bath}). Some examples in \cite{paper1, paper2} show that the analogous result  to Theorem~\ref{mainth} does not hold in general if the set of primes $\pi$ contains the prime $2$, although some related positive results were obtained in this case in \cite{paper2} when the factors are both soluble.

%
%Later on, in the search of new extensions of the Kegel and Wielandt's theorem through $\pi$-decomposable groups, the authors of this paper carried out in \cite{paper1,paper2, paper3, paper4} a longtime development with the aim to prove the following theorem, which was first stated as a conjecture in \cite{paper2}:
%\begin{teor}\textup{(\cite[Main Theorem]{paper4})}\label{mainth}
%Let $\pi$ be a set of odd primes. Let the group $G=AB$ be the product of two
%$\pi$-decomposable subgroups $A=A_\pi \times A_{\pi'}$ and $B=B_\pi \times B_{\pi'}$. Then
%$A_\pi B_\pi=B_\pi A_\pi$ and this is a Hall $\pi$-subgroup of $G$.
%\end{teor}
%This theorem, whose proof uses  the classification of finite simple groups (CFSG),  extends the mentioned previous ones of Berkovich, Arad and Chillag, Rowley, and Kazarin. Moreover, as a consequence we achieved also some non-simplicity and $\pi$-separability criteria for  products of $\pi$-decomposable groups, which contribute some new extensions  of the theorem of Kegel and Wielandt, as aimed.

It is well-known that, given a set of primes $\pi$, any soluble group or, more generally, any $\pi$-separable group, has a unique conjugacy class of Hall $\pi$-subgroups, and that every $\pi$-subgroup is contained in a Hall $\pi$-subgroup (dominance). So it is worthwhile emphasizing that the above development is closely related to a classic but ongoing problem in the theory of finite groups: the search for conditions which guarantee the  existence, conjugacy, and dominance of Hall subgroups in a finite group,  for a given set of primes $\pi$. To be more accurate, we will say that a group $G$ \emph{satisfies the property}:
\begin{itemize}\itemsep=0pt
\item[$E_{\pi}$]  if $G$ has at least one Hall $\pi$-subgroup;
\item[$C_{\pi}$] if $G$ satisfies $E_{\pi}$ and any two Hall $\pi$-subgroups of $G$ are conjugate in $G$;
\item[$D_{\pi}$] if $G$ satisfies $C_{\pi}$ and every $\pi$-subgroup of $G$ is contained in some Hall $\pi$-subgroup of $G$.
\end{itemize}
(Such a group is also called an \emph{$E_{\pi}$-group},  \emph{$C_{\pi}$-group}, and \emph{$D_{\pi}$-group}, respectively.)
After the seminal work of P. Hall (cf. \cite{Hall}), where he introduced the above terminology, numerous researchers have addressed the mentioned problem. Specially significant is a theorem of H. Wielandt \cite{Wie} which states that any group
possessing nilpotent Hall $\pi$-subgroups is a $D_{\pi}$-group (see Lemma \ref{Wie} below).  In particular, several authors have investigated in a series of papers the Hall subgroups of the finite simple groups. The classification of the simple groups satisfying the properties $E_{\pi}$, $C_{\pi}$, or $D_{\pi}$, has been completed by E. Vdovin and D. Revin. We refer to the expository article {\cite{VRsurvey} for a detailed account on this topic.

A natural question arises then for products of $\pi$-decomposable groups: What can be said about the properties $C_{\pi}$ and $D_{\pi}$ for such products of groups? For any set of odd primes, F. Gross proved in \cite{Gro2} that any group satisfying the $E_{\pi}$-property also satisfies the $C_{\pi}$ one. So, having in mind Theorem~\ref{mainth}, it remains to analyze the dominance property. This will be the principal aim of the present paper, which we will attain by proving  the following theorem:

\begin{mainth}\label{mainDpi} Let $\pi$ be a set of odd primes. Let the group $G=AB$ be the product of two
$\pi$-decomposable subgroups $A=A_\pi \times A_{\pi'}$ and $B=B_\pi \times B_{\pi'}$. Then $G$ satisfies the property $D_{\pi}$.
\end{mainth}

 In fact, our study on the $D_{\pi}$-property  was initially motivated by the development carried out in \cite{trif} on trifactorized groups and $\pi$-decomposability, where the dominance property appeared as a relevant tool.  \emph{Trifactorized} groups, that is, groups of the form $G = AB = AC= BC$, where $A$, $B$, and $C$ are subgroups of $G$,  play a key role within the study of factorized groups. For instance, the so-called  \emph{factorizer} of a normal subgroup in a factorized group turns out to be trifactorized (see \cite{AFG}). Specifically, for a subgroup $N$ of a group $G=AB$ which is the product of subgroups $A$ and $B$, the \emph{factorizer} of $N$ in $G$, denoted $X(N)$, is the intersection of all factorized subgroups of $G$ containing $N$. (A subgroup $S$ of $G=AB$ is \emph{factorized} if $S=(S\cap A)(S\cap B)$ and $A\cap B\le S$.)
In this setting, in \cite{trif} we proved the following result, which can be considered as a particular significant case of our Main Theorem:

\begin{teor}\textup{(\cite[Theorem 3.2]{trif})} \label{dpitri}
Let $\pi$ be a set of odd primes. Let the group $G=AB=AC=BC$ be the product of three
subgroups $A, B$ and $C$, where $A=A_\pi \times A_{\pi'}$ and $B=B_\pi \times B_{\pi'}$ are $\pi$-decomposable groups, and $C$ is a  $D_{\pi}$-group. Then $G$ is a  $D_{\pi}$-group.
\end{teor}

The notation is standard and is taken mainly from   \cite{D-H}. We also refer to this book  for the basic terminology and results about classes of groups.   If $X, Y$ are subgroups of a group $G$, we use the notation $X^{Y}=\langle x^y\mid x\in X, y\in Y\rangle$; in particular, $X^{G}$ is the normal closure of $X$ in $G$. Also, if $n$ is an integer and $p$ a prime number,  $n_p$ will denote the largest power of $p$ dividing $n$, and  $\pi(n)$ the set of prime divisors of $n$; for the order $|G|$ of a group $G$, we set $\pi(G)=\pi(|G|)$.

\section{Preliminary results}

The next result is a reformulation of a useful one due to Kegel, and later on improved by Wielandt, which appears in \cite[Lemma 2.5.1]{AFG} (see also \cite[Lemma 2]{paper2}).
\begin{lem}\label{2.5.1}
Let the group $G=AB$ be the product of the subgroups $A$ and $B$ and let $A_0$ and $B_0$ be normal subgroups of $A$ and $B$, respectively. If $A_0B_0=B_0A_0$, then $A_0^gB_0=B_0A_0^g$ for all $g\in G$.

Moreover, if $A_0$ and $B_0$ are $\pi$-groups for a set of primes $\pi$, and $O_\pi(G)=1$, then $[A_0^G,B_0^G]=1$.
\end{lem}

 We will use,  without further reference, the following fact on Hall subgroups of factorized groups, which is applicable to $\pi$-separable groups (see \cite[Lemma 1.3.2]{AFG}).

\begin{lem}\label{1.3.2}
Let $G=AB$ be the product of the subgroups $A$ and $B$. Assume that $A$ and  $B$ have Hall $\pi$-subgroups and that $G$ is a D$_{\pi}$-group for a set  of primes $\pi$. Then there exist  Hall $\pi$-subgroups $A_{\pi}$ of $A$ and $B_{\pi}$ of $B$ such that $A_{\pi}B_{\pi}$ is a Hall $\pi$-subgroup of $G$.
\end{lem}

The next well known result due to Wielandt will be relevant in our development.

\begin{lem}\textup{(\cite{Wie})}\label{Wie}
If a group $G$ has a nilpotent Hall $\pi$-subgroup, for a set of primes $\pi$, then $G$ is a $D_{\pi}$-group.
\end{lem}

We need specifically the following result, whose proof uses CFSG.

\begin{lem}\textup{(\cite[Theorem 7.7]{RV})}\label{RV} Let $G$ be a group, $A$ a normal subgroup of $G$, and $\pi$ be a set of primes. Then $G$ is a D$_{\pi}$-group if and only if $A$ and $G/A$ are D$_{\pi}$-groups.
\end{lem}

The following Lemma~\ref{primitive} on simple groups of Lie type will be essential for the proof of our main theorem.
We introduce first some additional terminology and notation. Let $n$ be a positive integer and $p$ be a prime number. A prime $r$ is said
to be \textit{primitive with respect to the pair $(p, n)$} if
$r$ divides $p^n-1$ but $r$ does not divide $p^k-1$ for every
integer $k$ such that $1\le k< n$. It was proved by
Zsigmondy~\cite{Zsi} that such a primitive prime $r$ exists  unless either $n=2$ and $p$ is a Mersenne prime or
$(p,n)=(2,6)$. For such a  prime $r$, it holds
 $r-1\equiv 0 \pmod n$ and, in particular, $r\geq n+1$.

In the sequel, for $q=p^e$, $e\ge 1$, we will denote by $q_n$ \emph{any} primitive  prime divisor of $p^{en}-1$, i.e. primitive with respect to the pair $(p, ne)$.

\begin{lem}\textup{(\cite[Lemma 14]{paper4})}\label{primitive}
For $N=G(q)$  a classical simple group of Lie type of characteristic $p$  and  $q=p^e$, and $N\unlhd G\le \Aut(N)$,   there exist primes $r, \, s \in \pi(N) \setminus \pi(G/N)$ and maximal tori $T_1$ and $T_2$ of $N$ as stated in Table 1.

Moreover, except for the case denoted  $(\star )$ in Table 1, for any element $a \in N$ of order $r$ and any element $b \in N$ of order $s$ we may assume  that   $C_{N}(a) \leq T_1$ and $C_{N}(b) \leq T_2$, and these are abelian $p'$-groups.

On the other hand, there is neither a  field automorphism nor a graph-field automorphism of $N$ centralizing elements of $N$ of order $r$ or $s$ (except for the triality automorphism in the case $P\Omega_{8}^{+}(q)$).
\end{lem}

\begin{table}[ht!]
$$\begin{array}{c|c|c|c|c|c}
 N& r &s& |T_1|& |T_2| & Remarks \\

\hline
& & & & & \\
 L_n(q)&  q_n &q_{n-1} & \frac{q^n-1}{(n, q-1)(q-1)}&\frac{q^{n-1}-1}{(n, q-1)}& (n, q) \neq (6, 2) \\
n \geq 4 & & & & & (n-1, q) \neq (6,2)\\
& & & & &  \\

\hline
& & & & & \\
 PSp_{2n}(q) &  q_{2n} & q_{2(n-1)} & \frac{q^{n}+1}{(2, q-1)}& \frac{(q^{n-1}+1)(q+1)}{(2, q-1)}&  n \mbox{ even } \quad (\star) \\
 & & & & & (n, q) \neq (4,2) \\
  P\Omega_{2n+1}(q) & & & & & \\
  &  q_{2n} & q_{n} & \frac{q^{n}+1}{(2, q-1)} &  \frac{(q^{n}-1)}{(2, q-1)}& n  \mbox{ odd} \\
  n\geq 3 & & & & & (n, q) \neq (3, 2)  \\
& & & & &  \\

\hline
& & & & & \\
P\Omega_{2n}^{-}(q)  & q_{2n} & q_{2(n-1)} & \frac{q^{n}+1}{(4, q^n+1)}& \frac{(q^{n-1}+1)(q-1)}{(4, q^n+1)}&  (n, q) \neq (4, 2) \\
n \geq 4 & & & & &  \\
& & & & &  \\

\hline
& & & & &  \\
P\Omega_{2n}^{+}(q) & q_{2(n-1)} & q_{n-1} &\frac{(q^{n-1}+1)(q+1)}{(4, q^n-1)}  &\frac{(q^{n-1}-1)(q-1)}{(4, q^n-1)} &  n \mbox{ even } \\
& & & & &  (n, q) \neq (4, 2)\\

  n \geq 4  & q_{2(n-1)} &  q_{n} &\frac{(q^{n-1}+1)(q+1)}{(4, q^n-1)} & \frac{q^n-1}{(4, q^n-1)} & n \mbox{ odd}\\
  & & & & &  \\
 \hline
 \end{array}$$
 \caption{Table 1}
 \end{table}

We refer to \cite{VV,VV2} for more accurate information about maximal tori in simple groups of Lie type.

\begin{lem}\textup{(\cite[Lemma 5]{paper4})}\label{outp} Let $N=G(q)$ be a classical simple group of Lie type over the field $GF(q)$ of characteristic $p$. Then $ |Out(N)|_p \leq q$ and equality holds only when  $q \in\{2, 3, 4\}$. Moreover, if $q=3$, the only case in which  $|Out(N)|_p = q$ is possible is when $N=P\Omega^+_8(q)$.
In particular, $|Out(N)|_p < q^2$ for any classical simple group of Lie type.
\end{lem}

We will use the main results in \cite{LPS} about factorizations of the simple groups and their automorphisms groups by their  maximal subgroups. Also the information from \cite{VRsurvey} about simple groups satisfying property $D_{\pi}$ will be needed in some cases.

\section{The minimal counterexample}

We describe next the structure of a minimal counterexample to our Main Theorem.

\begin{pro}\label{mincount}
Let $\pi$ be a set of odd primes. Assume that  the group $G=AB$ is the product of two
$\pi$-decomposable subgroups $A=A_{\pi} \times A_{\pi'}$ and $B=B_{\pi} \times B_{\pi'}$ and $G$ is a counterexample of minimal order to the assertion that $G$ satisfies property $D_{\pi}$.

Then $G$ has a unique minimal normal subgroup $N$, which is a non-abelian simple group,  so that $N \unlhd G \leq \Aut(N)$.

 Moreover,  the following properties hold:
\begin{enumerate}

\item[(1)] $G=AN=BN=AB$.

In particular, $|N||A \cap B|=|G/N||N\cap A||N \cap B|$.
\item[(2)]  One of the subgroups, say $B$, is a $\pi'$-group, i.e. $B_{\pi}=1$. Hence, $\pi(G/N) \subseteq \pi'$.
\item[(3)]  $A$ is neither a $\pi$-group nor a $\pi'$-group, i.e. $A_{\pi} \neq 1$ and $A_{\pi'}  \neq 1$.
\item[(4)] $A_{\pi}$ is a Hall $\pi$-subgroup of $N$ (and of $G$), so $|\pi(A)| \geq 3$.

In particular, $N$ satifies $E_\pi$, but not $D_{\pi}$ neither $E_{\pi'}$, and so $|\pi \cap \pi(N)| \geq 2$ and $|\pi' \cap \pi(N)| \geq 2$.
\item[(5)]  The subgroup $A_\pi$ is not nilpotent.

\item[(6)] Let $S \leq A$ be an $s$-group, for a prime number $s$. If $s \in \sigma$, with $\sigma \in \{\pi, \pi'\}$, then $\pi(|A:C_{A}(S)|) \subseteq \sigma$ and $C_A(S)$ is not a $\sigma$-group.

\item[(7)]  Either $A$ or $B$ is non-soluble.

Moreover, for each prime $r \in \pi$ and any other prime  $s \in \pi(A \cap N)$, there exists a soluble subgroup of $N$ of order divisible by $r$ and $s$.

\item[(8)]  $|\pi' \cap \pi(N)| \geq 3$, and $|\pi(N)| \geq 5$. Moreover, $\pi(G)=\pi(N)$.

\item[(9)] Assume that $N$ is a simple group of Lie type over the field $GF(q)$ of characteristic $p$. Then the following assertions hold:
\begin{itemize}
\item[a)] $p \in \pi'$.
\item[b)] If $N$ is of Lie rank $l > 1$, then $p \in \pi(B \cap N)$.
\item[c)] Assume that there exists an element $ a \in A \cap N$  of prime order $r$ and a subgroup $X$ in $G$  containing $A$ such that $C_{X \cap N}(a)$ is an abelian $p'$-subgroup.  Then $r \in \pi$, $A$ is soluble, $A \cap N$ is a $p'$-group and $|A|_p \leq q$.

If, in addition, $C_{X}(a)$ is a $p'$-group, then $A$ is a $p'$-group.

\end{itemize}
\end{enumerate}
\end{pro}
\begin{proof}

We prove first statements (1)-(4).

Recall that  $A_{\pi}B_{\pi}$ is a Hall $\pi$-subgroup of $G$ by Theorem~\ref{mainth}.

 Notice that the hypotheses of the result hold for factor groups. Hence whenever $N$ is a non-trivial normal subgroup of $G$, the minimality of $G$  implies that $G/N$ is a $D_{\pi}$-group. If in addition $N$ were a $D_{\pi}$-group, then $G$ would be a $D_{\pi}$-group by Lemma~\ref{RV}, which is not the case.
In particular, we deduce that $O_\pi(G)=O_{\pi'}(G)=1$.

By Lemma~\ref{2.5.1}, it follows that $[A_\pi^G, B_\pi ^G]= 1$.

We consider now the case that  $A_\pi\neq 1$ and $B_\pi \neq 1$. We notice that $A_\pi^G\cap B_\pi^G=1$. Otherwise, if $N$ is a minimal normal subgroup of $G$ contained in $A_\pi^G\cap B_\pi^G$, then $[N,N]=1$, i.e. $N$ is abelian, and then either $N\le O_\pi(G)=1$ or $N\le O_{\pi'}(G)=1$, a contradiction.

Let $H$ be a $\pi$-subgroup of $G$. We aim to prove that $H\le (A_{\pi}B_{\pi})^g$ for some $g\in G$, to get a contradiction and derive that either $A_\pi=1$ or $B_\pi=1$.
By minimality of $G$, both factor groups $G/A_\pi^G$ and $G/B_\pi^G$ satisfy the property $D_{\pi}$, and so we may assume:
\begin{gather*}
H \leq A_{\pi} B_{\pi}A_\pi^G = B_{\pi} A_{\pi}^G \\
H \leq (A_{\pi} B_{\pi}B_{\pi}^G)^{g} = A_{\pi}^{b} B_{\pi}^G
\end{gather*}
for some $g=ab$ with $a\in A$, $b\in B$, since $B_{\pi}^{G}$ is normal in $G$ and $A_\pi$ is normal in $A$.  Now applying that $A_{\pi}^{G}$ is normal in $G$ and $B_\pi$ is normal in $B$, we deduce:
\begin{gather*}
H \leq B_{\pi} A_{\pi}^G \cap A_{\pi}^{b} B_{\pi}^G = (B_{\pi} A_{\pi}^G  \cap A_{\pi} B_{\pi}^G)^{b}= \\
 (A_{\pi}B_{\pi} (A_{\pi}^{G} \cap B_{\pi}^{G}))^{b}=(A_{\pi}B_{\pi})^{b}
\end{gather*}
as aimed. So, without loss of generality, we can consider  $B_{\pi}=1$. Also, if  $G$ would satisfy $E_{\pi}$ and $E_{\pi'}$, then $G$ would be a $D_{\pi}$-group (see \cite{A-F}), a contradiction. Hence $A_{\pi} \neq 1$ and $A_{\pi'} \neq 1$.

Let $N$ be a minimal normal subgroup of $G$, and let $X(N)$ be the factorizer of $N$ in $G=AB$. Assume that $X(N)$ is a proper subgroup of $G$. Then the minimality of $G$ implies that $X(N)$ satisfies $D_\pi$. Since $N\unlhd X(N)$ we have that $N$ is a $D_\pi$-group. But also $G/N$ satisfies $D_{\pi}$. Hence we get that $G$ is a $D_{\pi}$-group by Lemma~\ref{RV}, a contradiction. Therefore $X(N)=G=AB=AN=BN$. By order arguments we have that $A_\pi\le N$.  If $M$ were another minimal normal subgroup of $G$, then it would be $A_\pi\le N\cap M=1$, a contradiction. Consequently, $G$ has a unique minimal normal subgroup, say $N$,  which is non-abelian.

Hence, $N\cong S\times \cdots \times S$ is a direct product of copies of a non-abelian simple group $S$. Since $G= A_{\pi'}N$, it follows that $A_{\pi'}$ permutes transitively the non-abelian simple components of $N$ and centralizes $A_\pi$, which implies that $N$ is a simple group.

Consequently, $G=AB=AN=BN$ is an almost simple group, i.e. $N \leq G \leq  \Aut(N)$ with $N$ a non-abelian simple group, $A_\pi \neq 1$, $A_{\pi'} \neq 1$, $B_\pi = 1$ and $B_{\pi'} \neq 1$. In particular, $A_{\pi} \leq N$, and $A_{\pi}$ is a Hall $\pi$-subgroup of $N$, and also of $G$, since $G/N=BN/N$ is a $\pi'$-group. Moreover, $N$ is not a $D_{\pi}$-group by Lemma~\ref{RV}. In particular, $|\pi| \geq 2$ and $|\pi(A)| \geq 3$. Also,  $N$ is not an $E_{\pi'}$-group, since it is an $E_{\pi}$-group (see \cite{A-F}). A straightforward computation shows that $|N||A \cap B|=|G/N||N\cap A||N \cap B|$. Hence statements (1)-(4) hold.

\medskip

\noindent (5) Assume that $A_\pi$ is nilpotent. Then, since this is a Hall $\pi$-subgroup of $G$, it follows that $G$ is a $D_{\pi}$-group, by Lemma \ref{Wie}, a contradiction.

\medskip

\noindent (6) Assume that $S$ is an $s$-subgroup of $A$, with $s \in \sigma$, and either $\sigma =\pi$ or $\sigma =\pi'$. The first statement  is clear since $A_{\sigma'} \leq C_A(S)$.
Consequently, if $C_{A}(S)$ were a $\sigma$-group, $A$ would be a $\sigma$-group, a contradiction by (3).

\medskip

\noindent (7) Assume now that both $A$ and $B$ are soluble. From \cite{Ka2}  it is known that $N$ should be isomorphic to one of the groups in the set:
$$ \mathfrak{M} =\{ L_2(q), q > 3; L_3(q), q < 9; L_4(2), M_{11}, \mbox{PS}p_4(3), U_3(8) \}.$$
All factorizations for such a group $G=AB$, with $A$ and $B$ soluble, can be found in \cite[Proposition 4.1]{LiX}.  In fact, as stated in this reference, the groups $L_3(5)$, $L_3(7)$ and $L_4(2)\cong A_8$ can be omitted.
For any of the groups in $\mathfrak{M} \setminus \{ L_2(q), q > 3;$ $ M_{11}, L_3(5), L_3(7), L_4(2) \}$, we get a contradiction with the assertion in (4) that $|\pi \cap \pi(N)| \geq 2$ and $|\pi' \cap \pi(N)| \geq 2$. In the cases in which $|\pi(N)|=4$, this is  deduced from the fact that $N$ has a self-centralizing Sylow $s$-subgroup for some large prime $s \not\in \pi (A) $ by (6), and taking into account the factorizations of $G$ and the fact  that $\pi(G)=\pi(N)$ in those cases.

 The case $N=M_{11}$  is discarded because, if it satisfies $E_{\pi}$ for a set of odd primes, then it satisfies $D_{\pi}$ (see \cite{VRsurvey}).

Consider now the case $N=L_2(q)$, with $q > 3$. If $p$ is the characteristic, then $ p \not\in \pi$ by (6). From \cite{LPS}
it follows that, apart from some exceptional cases with $q \in \{7, 11, 23 \}$ which can be discarded by similar arguments as used above, the maximal soluble subgroups $X$ and $Y$ of $G$ such that $G=XY$ satisfy the condition $\{X \cap N, Y\cap N \}=\{N_N(P), D_{\nu(q+1)}\}$, with $P$ a Sylow $p$-subgroup of $N$, $|N_N(P)|=\epsilon q(q-1)$ with  $\epsilon=(q-1, 2)^{-1} $, and $D_{\nu(q+1)}$ a dihedral group of order $\nu(q+1)$ with $\nu=(2, p)$. The case that $A \cap N \leq D_{\nu(q+1)}$ can be discarded because this means that $A_{\pi}$ is abelian, which contradicts (5). In the other case, when $A \cap N \leq N_{N}(P)$, since the centralizer of any $p$-element in $N$ is a $p$-group, and $p$ should divide $|N \cap A|$ by order arguments as $B \cap N \leq D_{\nu(q+1)}$, it follows that $A_{\pi}=1$, a contradiction.

Hence, either $A$ or $B$ is non-soluble.

Now, assume that $r \in \pi$ and $s \in \pi(A \cap N)$. If $s \in \pi$, then $A_{\pi}$ is a soluble subgroup  of $N$ of order divisible by $r$ and $s$. If $s \in \pi'$, we can consider the soluble subgroup $A_{\pi} \times (A_{s}\cap N)$ for $A_{s}$ a Sylow $s$-subgroup of $A$. Hence  (7) holds.

\medskip

\noindent (8) From (7), we deduce that  either $N \cap A$ or $N \cap B$ is non-soluble, because $G/N \cong A/(N \cap A)\cong B/(B\cap N)$ is soluble. Hence $|\pi' \cap \pi(N)| \geq 3$ and $|\pi(N)| \geq 5$.
Our next aim is to prove that $\pi(G)=\pi(N)$. Assume that $ \sigma:= \pi(G) \setminus \pi(N) \neq \emptyset$, and take a prime number $s \in \sigma$. Note that $\sigma \subseteq \pi'$, since $N$ contains $A_{\pi}$, which is a Hall $\pi$-subgroup of $G$.
We have $|N||A \cap B|=|G/N||N \cap A||N \cap B|$. Hence $|G: A \cap B|=|N|^{2}/|N \cap A||N \cap B|$, and so $(s,  |G: A \cap B|)=1$.  Then $A\cap B$ contains a Sylow $s$-subgroup of $G$, say $S$.

Let $ \pi_{0}:=\pi' \setminus \{s \}$. Again since $G/N \cong A/(N \cap A)$ is a soluble $\pi'$-group, we may choose a Hall  $\pi_{0}$-subgroup of $A$, say  $\tilde{A}$, such that $A_{\pi'}=\tilde{A}S$, and so $A= A_\pi \times\tilde{A}S$. Let $\tilde{G}:=\tilde{A} N= (A_{\pi} \times \tilde{A}) N $. Consider now $\tilde{B}:=B \cap \tilde{G}=B \cap \tilde{A}N$. Since $A_{\pi'}N=BN$ and $B$ also contains $S$ we can deduce that $B=B\cap A_{\pi'}N=S(B \cap \tilde{A}N)=\tilde{B}S$ and $\tilde{B}$ is a Hall $\pi_0$-subgroup of $B$. Moreover $(A_{\pi} \times \tilde{A}) \cap B=\tilde{A}\cap B=\tilde{A} \cap \tilde{B}$. Since $(|S|,|N \cap A|)=1=(|S|,|N \cap B|)$ it is easy to see that $|\tilde{G}|= |G|/|S|=|(A_{\pi} \times \tilde{A})S\tilde{B}|/|S|=|A_{\pi} \times \tilde{A}||\tilde{B}|/|(A_{\pi} \times \tilde{A})\cap \tilde{B}|=|(A_{\pi} \times \tilde{A})\tilde{B}|$. Hence $\tilde{G}=(A_{\pi} \times \tilde{A})\tilde{B}$ is a subgroup of $G$, which is a product of two $\pi$-decomposable groups, and contains $N$. If $\tilde{G} < G$, then by the choice of $G$ we deduce that $N$ is a $D_{\pi}$-group, which is a contradiction. This implies that $\sigma=\emptyset$ and the assertion (8) follows.

\medskip

\noindent (9) Assume finally that $N$ is a simple group of Lie type of characteristic $p$.
If $P$ is any Sylow $p$-subgroup of $N$, then $C_{\Aut(N)}(P)$ is a $p$-group, so it follows from (6) that $p \in \pi'$. Now, if $N$ has Lie rank $l > 1$, arguing as in the proof of \cite[Lemma 12]{paper4}, we can deduce that $p \in \pi(B \cap N)$.

\smallskip

Finally, assume that there exists an element $ a \in A \cap N$  of prime order $r$ and a subgroup $X$ of $G$ such that  $A \leq X$ and $C_{X \cap N}(a)$ is an abelian $p'$-subgroup.
If $r \in \pi'$, then $A_{\pi} \leq C_{X \cap N}(a)$ and hence $A_{\pi}$ is abelian, a contradiction with (5). Then $r \in \pi$, and  $A_{\pi'} \cap N \leq C_{X \cap N}(a)$ is an abelian  $p'$-group. It follows that $A \cap N$ is soluble, and so is $A$. Moreover, since $p \in \pi'$,  $A \cap N$ is a $p'$-group.  In this case $|A|_p =|AN/N|_p=|G/N|_p \leq q$ by Lemma~\ref{outp}. The last assertion is straightforward. Hence (9) holds.
\end{proof}

For a group $G$ as in Proposition~\ref{mincount}, and its unique minimal normal subgroup $N$, we have the following results:

\begin{lem}\label{sporadic}
$N$ is not a sporadic simple group.
\end{lem}
\begin{proof}

By \cite[Theorem C]{LPS} if $N$ is a sporadic simple group, $N \unlhd G \leq \Aut(N)$,  and $G$ is factorized, we have that $$N \in \{M_{11}, \, M_{12}, \, M_{22}, \, M_{23}, \, M_{24}, \, J_2, \, HS, \, He, \, Ru, \, Suz, \, Fi_{22}, \, Co_1\}.$$

Now, by comparing \cite[Theorem 8.2~(Table 3)]{VRsurvey} with \cite[Theorem 6.9. Condition II]{VRsurvey}, we can deduce that any of the above groups being an $E_{\pi}$-group for a set of odd primes, is also a $D_{\pi}$-group. So this case is not possible.
\end{proof}

\begin{lem}\label{alternating}
$N$ is not an alternating group of degree $n \geq 5$.
\end{lem}
\begin{proof}
By \cite[Theorem 8.1]{VRsurvey},  if  $\pi$ is a set of odd primes with $|\pi| > 1$, then any alternating group is not an $E_{\pi}$-group. So this case can be ruled out.
\end{proof}
\begin{lem}\label{exceptional}
$N$ is not an exceptional group of Lie type.
\end{lem}
\begin{proof} By \cite[Theorem B]{LPS}, if $N$ is an exceptional group of Lie type, $N \unlhd G \leq \Aut(N)$ and $G$ is factorized, then $$N \in \{G_2(q),q=3^c;  \, F_4(q), q=2^c; G_2(4)\}.$$
If $N = G_2(q)$,  then $N$ is a $D_{\pi}$-group for any set of primes $\pi$ such that $2, p \not\in\pi$ by \cite[Theorem 6.9. Condition II]{VRsurvey}, a contradiction.

Let $N = F_4(q), q=2^c$. In this case all possible factorizations $G=AB$ (not only the maximal ones) with subgroups $A, B$ not containing $N$ are as follows: For $\{X, Y\}=\{A,B\}$, $X \cap N =Sp_8(q)$ and $Y \cap N \in \{{}^{3}D_4(q), ^{3}D_4(q).3 \}$ and $N=(A \cap N)(B \cap N)$. Since 2 divides $( |A \cap N|, |B \cap N|)$ and each of these subgroups has a Sylow $2$-subgroup containing its centralizer in the corresponding subgroup, it follows that $N$ is a $\pi'$-group, which is a contradiction.\end{proof}

\subsection{The almost simple case for classical groups of Lie type}
%The remainder cases to be analyzed in order to prove our Main Theorem are classical groups of Lie type.
%\begin{pro}\label{factLPS}
%Let $N$ be a classical simple group of Lie type over the field $GF(q)$ of characteristic $p$, and let $G$ be a group such that $N \unlhd G \leq \Aut(N)$. Suppose that $G$ has a non-trivial factorization $G=XY$ with $X$ and $Y$ maximal subgroups of $G$ not containing $N$. Then $N$ belongs to one of the following families of groups:
%\begin{enumerate}
%
%\item[(1)] $N=L_n(q), n\ge 2, (n,q)\ne (2,2)$ or $(2,3)$;
%
%\item[(2)] $N=PSp_{2m}(q), m\ge 2$;
%
%\item[(3)] $N=U_{2m}(q), m\ge 2$;
%
%\item[(4)] $N=\Omega_{2m+1}(q), m\geq 3$, $q$ odd;
%
%\item[(5)] $N=\Omega_{2m}^-(q), m\geq 4$;
%
%\item[(6)] $N=\Omega_{2m}^+(q), m\geq 4$;
%
%\item[(7)] $N$ is one of the groups with special factorizations described in Tables 2 and 3 in \textup{\cite{LPS}}.
%\end{enumerate}
%\end{pro}

In the sequel, $G=AB$ will be a minimal counterexample to the Main Theorem, which is an almost simple group with socle $N$ (i.e. $N \leq G \leq \Aut(N)$), and $N$ is a classical simple group of Lie type over a field $GF(q)$ of prime characteristic $p$,  with $q=p^e$,  $e$ a positive integer. Information about the  structure of such $G$ is collected in Proposition~\ref{mincount}. In particular, recall that $G=AN=BN=AB$ and  $B$ is a $\pi'$-group.

For such a group $G$, we will use the knowledge about the non-trivial factorizations  $G=XY$, where $X$ and $Y$ are maximal subgroups of $G$ not containing $N$, which appears in \cite[Theorem A, Tables 1-4]{LPS}. In the referred tables and the corresponding proofs an explicit description of the ``large'' subgroups $X_0 \unlhd X \cap N$, $Y_0 \unlhd Y \cap N$  is given (in most cases, $X_0 = X \cap N$ and $Y_0 = Y \cap N$).  For any appearence of a pair $(X,Y)$, we will distinguish two subcases: either $A\leq X, B\leq Y$, or   $A\leq Y, B \leq X$.

Note that if we  assume, for instance,  that $A\leq X, B\leq Y$, then the following facts hold:

- $G=AY=BX$.

- $G=XN=YN$.

- $\pi\subseteq \pi(X)$, and $A_{\pi}$ is a Hall $\pi$-subgroup of $X$.

- $|G:Y|=|N: N \cap Y|$ divides $|A|$.

- $|G:X|=|N:N \cap X|$ divides $|B|$.

- $X=A(B \cap X)$, $Y=B(A \cap Y)$, and both $X$, $Y$ satisfy $D_{\pi}$.

\medskip

We adher also to the notation in \cite{LPS} for the almost simple groups and its subgroups. In particular, if $\hat{N}$ is a classical linear group on the vector space $V$ with centre $Z$ (so that $N=\hat{N}/Z$ is a classical simple group) and $\hat{N} \unlhd \hat{G} \leq GL(V)$, for any subgroup $X$ of $\hat{G}$ we will denote by $\, \hat{}X$ the subgroup $(XZ \cap \hat{N})/Z$ of $N$. Also we will use the notation $P_i$, $1 \leq i \leq m$, $N_{i}$, $N_{i}^{\epsilon}$ $(\epsilon=\pm)$, for stabilizers of subspaces as described in \cite[2.2.4]{LPS}.

We recall that  for a prime $p$ and $q=p^e$, $e\ge 1$,  we denote by $q_n$ \emph{any} primitive  prime divisor of $p^{en}-1$, i.e. primitive with respect to the pair $(p, ne)$.

We will use frequently, without further reference,  the fact that if $N$ is a simple group of Lie type over $GF(q)$, $q = p^e$,  $n \geq 3$ and $(q,n) \neq (2,6)$, then $q_n$   does not divide $|Out(N)|$ (see \cite[2.4 Proposition B]{LPS}). Also the information about elements whose order is a primitive prime divisor appearing in Lemma~\ref{primitive} will be used eventually without reference.

\smallskip
In the proof of the next lemmas we will refer to Proposition~\ref{mincount}(i) just as 1(i).

\begin{lem}\label{linear} $N$ is not isomorphic to $L_{m}(q)$.
\end{lem}
\begin{proof}
We do not consider here the cases when $L_m(q)$ is isomorphic to an alternating group, which  have been discarded in Lemma~\ref{alternating}.

Recall that $|\pi(N)| \geq 5$ by 1(8).

\smallskip

 We analyze first the cases $m \leq 3$. If $N=L_2(q)$, all possible factorizations appearing in \cite[Tables 1-3]{LPS} can be ruled out, except when $q=29$ or $q=59$,  because either $|\pi(N)| \leq 4$ or both $X$ and $Y$ are soluble. When $q \in \{29, 59\}$, there exists also a factorization $G=XY$  with $X \cap N=P_1$, $Y \cap N = A_{5}$, the alternating group of degree $5$.  In both cases it should be $q \in \pi'$ by 1(9), $\{3, 5 \} \subseteq \pi$  and $A \leq A_{5}$. But the fact that $A_5$ do not have Hall $\{3,5\}$-subgroups leads to a contradicton.

 Consider now  $N = L_3(q)$, so $|N|=(1/(q-1,3))q^3(q^3-1)(q^2-1)$ and $|Out(N)|= 2(q-1,3)e$. The cases when $q \leq 8$ are excluded by 1(8). From \cite{LPS}  we know that all factorizations $G=XY$ satisfy that for one of the factors, say $X$,  $|N \cap X|$ divides $\frac{q^3-1}{q-1} \cdot 3$, which is  not divisible by $p$ if $p \neq 3$. Since $p \in \pi(N \cap B)$ by 1(9) this forces that $A \leq X$ when $p\neq 3$. For the case $p=3$, we get that $| N : N \cap Y|_3 \leq q/3 < q^2$, so $C_G(N \cap Y)$ is a $p$-group because of \cite[Lemma 6]{paper4}, and this also implies that $A \leq X$. Since $q^3-1$ divides $|G:Y|$ we get that $r=q_3 \in \pi(A)$. But from the known structure of $X \cap N $ (this group is the normalizer of a Singer cycle of order $(q^3-1)/(q-1)$), it is clear that an element of order $3$ cannot centralize an element of order dividing $(q^3-1)/(q-1)$. This means that $A_{\pi} \leq A \cap N$ should be contained in a Singer cycle, and so it is abelian, which contradicts 1(5).
 \medskip

Assume from now on that $m \geq 4$. In this case both $q_m$ and $q_{m-1}$ exist, except when $(m,q)=(6,2)$ or $(7,2)$.
\smallskip

Take first $N=L_6(2)$, so $|N|=2^{15}\cdot 3^4 \cdot 5 \cdot 7^2 \cdot 31$ and $|Out(N)|=2$. This group has a Sylow $31$-subgroup which is self-centralizing in $G$. This forces that $31 \in  \pi(B) \subseteq \pi'$. Hence, according to \cite[Table 1]{LPS}, it should be  $B \leq Y$ with $Y \cap N \in \{ P_1, P_5 \}$, and $A \leq X$, with $X \cap N \in \{ \, \hat{}GL_{a}(2^b), ab=6, b \mbox{ prime}; PSp_6(2) \}$. In all cases $2^6-1=7 \cdot 3^2$ divides $|A|$. The cases $X \cap N=\hat{}GL_{3}(2^2).2$ and $X \cap N=PSp_6(2)$ can be ruled out by applying that a Sylow 7-subgroup of such a group is self-centralizing and $7 \in \pi(A) \cap \pi(B) \subseteq \pi'$, which forces that $A$ is a $\pi'$-group. If $X \cap N= \hat{}GL_{2}(2^3).3$, by order arguments it should be $\{3, 5 \} \subseteq \pi(B)$ and so $\pi=\{7\}$, a contradiction.

Take now $N=L_7(2)$, so $|N|=2^{21}\cdot 3^4 \cdot 5 \cdot 7^2 \cdot 31 \cdot 127$ and $|Out(N)|=2$. In this case $r=q_m=127$ and $N$ has a self-centraling Sylow $r$-subgroup, which forces that $r \in \pi(B)$. Then, by \cite[Table 1]{LPS}, it should be  $A \leq Y$ with $Y \cap N=P_1$ or $ P_6$, and $B \leq X$ with $X \cap N=\hat{}GL_{1}(2^7).7=127.7$, a contradiction with 1(9) because $p=2$ should divide $|N \cap B|$.
\medskip

So we may assume from now on that  both $q_m$ and $q_{m-1}$ exist. We analyze next the different factorizations $G=XY$ appearing in  \cite[Table 1]{LPS}.

\medskip

\noindent \textbf{Case  $X \cap N= \hat{}GL_a(q^b).b$} (with $m=ab$,  $b$ prime),  $Y \cap N= P_1$ or $P_{m-1}$.

Assume first that $A \leq X$ and $B \leq Y$. Then $r=q_m$ divides $|G:Y|$, and so $q_m\in \pi(A)$. If we take an element $a \in A \cap N$ of order $r$ it holds that $C_{N}(a)$ is contained in a torus of order $(q^{m}-1)/((m,q-1)(q-1))$. Hence, by 1(9), we deduce that $r\in \pi$, and $A$ is soluble. From  \cite[Lemma 2.5]{ACK} the order of a soluble subgroup of $N$ whose order is divisible by $r$ should divide  either $m(q^{m}-1)$,  or $2m^2(m+1)(q-1)l$ with  $q=p$, $r=m+1$, $m=2^{l}$. But in this last case, $r=m+1$ is the only primitive prime divisor of such group with respect to the pair $(q, m)$ and so, applying \cite[2.4 Proposition D]{LPS} it should be $(q, m) \in \{(2,4), (2,10), (2, 12), (2, 18), (3,4), (3, 6), (5,6 )\}$,  but this contradicts the fact that $m$ is a power of $2$ and $|\pi(G)| \geq 5$. Hence $A \cap N$ is a soluble subgroup of order dividing $m(q^{m}-1)$. Indeed, by \cite[Theorem 1.1]{LiX}, such a soluble group should be contained in $\hat{}GL_1(q^m).m$, the normalizer of a Singer cycle in $N$. Hence, from the known structure of this group,  we get that $A_\pi \leq A \cap N$ is contained in a Singer cycle and so it is abelian, which contradicts 1(5).

Let suppose now that $A \leq Y$, $B \leq X$. In this case $q^{m-1}-1$ divides $|G:X|$ and so $s=q_{m-1} \in \pi(A)$. If we consider an element $b$ of order $s$ in $A \cap N$ it holds by Lemma \ref{primitive} that $C_{N}(b)$ is contained in an abelian subgroup of order $(q^{m-1}-1)/(q-1,m)$. Therefore,  by 1(9), it holds that  $s\in \pi$, $A \cap N$ is a $p'$-group and $|A|_p \leq q$. But this is a contradiction since  $|G:X|_p \geq q^2$, and $|G:X|$ divides $|A|$.

\medskip
\noindent \textbf{Case $X \cap N=PSp_m(q).a$}, $ m=2k$ even, $m\geq 4$, $a \leq 2$. Here $Y \cap N=P_1, P_{m-1}$ or $Y \cap N=Stab(V_1\oplus V_{m-1})$ (and in this case $G$ contains a graph automorphism).

Assume first that $A\leq X$. Then $q^{m}-1$ divides $|G:Y|$, and so we deduce that $r=q_m\in \pi(A)$. Since $m=2k$ is even, the centralizer of an element of order $r$ in $X \cap N$ is an abelian group of order $\frac{q^k+1}{(2, q-1)}$, by Lemma~\ref{primitive}. By 1(9), we can deduce that   $r\in \pi$ and $A$ is  soluble. But by \cite[Lemma 2.8]{ACK} the order of a soluble subgroup of $PSp_m(q)$ whose order is divisible by $r$ should divide  either $2k(q^k+1)$,  or $16k^2(q-1)r log_2(2k)$ with  $q = p$,
$r =2k+1=m+1$, and $k$  a power of $2$. In the last case, it holds that $r$ is the only primitive prime divisor of $q^{m}-1$, and so by \cite[2.4 Proposition D]{LPS} we get a contradiction as above. Hence
$|A \cap N|$ should divide $2k(q^k+1)$, which leads to a contradiction by order arguments since  $q^{m}-1$ divides $|A \cap N|$.

Assume now that $B\leq X$ and  $A\leq Y$. Then
$q^{m-1}-1$ divides $|G:X|$ and so  $|A|$. This means that $s=q_{m-1}\in \pi(A)$ and  applying 1(9) and arguing as above we get a contradiction, because here also $|G:X|_p \geq q^2$.

\medskip

\noindent \textbf{Case  }$X \cap N=\hat{}GL_{m/2}(q^2).2$, $Y \cap N=Stab(V_1\oplus V_{m-1})$, with  $q\in\{2,4\}$, and $m \geq 4$ even (here $G$ contains a graph automorphism).

Assume first that $A\leq X$, $B \leq Y$. Clearly, $|G:Y|$ is divisible by $r=q_m$, so that $r\in \pi(A)$. The centralizer in $N$ of an element $a$ of order $r$ in $N\cap A$ is contained in a torus of order $(q^{m}-1)/((m,q-1)(q-1))$. Hence, by 1(9), we deduce that $r\in \pi$, $A \cap N$ is a $2'$-group and $ |A|_2 \leq q$, which is a contradiction again by order arguments.

If $A\leq Y$, $B \leq X$, then $s=q_{m-1}\in \pi(A)$ and arguing as above we get a contradiction.

\medskip

Finally, from \cite[Table 3]{LPS} for the case $m=5$, $q=2$ there exists also a factorization $G=XY$ with $X \cap N=31.5$ and $Y \cap N= P_2$ or $P_3$. Since a Sylow 31-subgroup of $N$ is self centralizing in $G$, it should be $31 \not\in \pi(A)$ and so $ A \leq Y$. But then $ N \cap B$ should be a $p'$-group, which contradicts 1(9).

The lemma is proved. \end{proof}

\begin{lem}\label{Un}
$N$ is not isomorphic to $U_m(q)$, $m\geq 3$.
\end{lem}
\begin{proof}
Assume that $N = U_{m}(q)$, $m \geq 3$. Recall that  $p \in \pi'$, by  1(9).

\medskip

\noindent\textbf{Case $m$ odd}. From  \cite[Theorem A, Corollary 2]{LPS}, the only groups $G$ such that $N \leq G \leq \Aut(N)$ which are factorizable appear for $N = U_3(3)$, $U_3(5)$, $U_3(8)$ or $U_9(2)$.
By 1(8), $|\pi(N)| \geq 5$. Thus the only possible case would be $N = U_9(2)$. Note that $|N|=2^{36}\cdot 3^{11}\cdot 5^2\cdot 7\cdot 11\cdot17\cdot 19\cdot 43$ and $|Out(N)|=6$.
From \cite[Table 3]{LPS}, there exists a unique factorization of this group such that $X \cap N=J_3$ and $Y \cap N= P_1$, where $P_1$ is a parabolic maximal subgroup  of $N$.

Assume first that $A \cap N \leq X \cap N= J_3$. Note that $|J_3|=2^7\cdot 3^5\cdot 5\cdot 17\cdot 19$. In this case,  $\{ 3, 5\} \subseteq \pi(B) \subseteq \pi'$, by order arguments. Hence $\pi \subseteq \{17, 19\}$. But $J_3$ has no subgroups of order $17\cdot 19$ (see \cite[p. 82]{Atlas}), so we get a contradiction. Therefore, we may assume that $A \cap N \leq Y\cap N=P_1$, a parabolic subgroup. In such case, $\{2, 3, 5, 17, 19\} \subseteq \pi' $, and a Hall $\pi$-subgroup of $N$  has  order dividing $7\cdot 11\cdot 43$. But  there are no subgroups of order  $7\cdot 11$ in $U_9(2)$ (see for example \cite{VV}), and any subgroup of order $43$ is self-centralizing, so this case cannot occur.

\medskip

\noindent\textbf{Case $m$ even}, $m=2k$, $k\geq2$. It follows from \cite[Tables 1, 3]{LPS} (and with the same notation)
that  one of the maximal subgroups in the factorization of $G$ with $N \leq G \leq \Aut(N)$, say $X$, has the property $X \cap N =N_1 = U_{2k-1}(q)$, unless $N = U_4(2)$ or $U_4(3)$. Since $|U_4(2)|=2^6\cdot 3^4\cdot 5$ and $|U_4(3)|=3^6\cdot 2^7\cdot 5\cdot 7$, these possibilities are excluded by 1(8).

Apart from some exceptional cases, that we check below, any subgroup $H$ such that $N_1\leq H\leq \Aut(N_1)$ has no proper factorization with factors not containing $N_1$.

Assume  that $A \leq X$, and so  $X=A(X \cap B)$. Now note that $X=N_G(N_1)$, so $X/C_{G}(N_1)$ is isomorphic to a subgroup of $\Aut(N_1)$ and then it has no proper factorizations.   If $N_1 \cong N_1C_{G}(N_1)/C_{G}(N_1)$ were contained either in the $\pi$-decomposable group $AC_{G}(N_1)/C_{G}(N_1)$ or in the $\pi'$-group $(X \cap B)C_{G}(N_1)/C_{G}(N_1)$, it would follow that $N_1 = X \cap N$ would be  a $\pi'$-group, a contradiction as $A\cap N\le N_1=X\cap N$. Hence it holds that either $X=AC_G(N_1)$ or $X=(X\cap B)C_G(N_1)$, and we can argue like in the proof of \cite[Lemma 17]{paper4} to get also a contradiction.

The exceptional cases, when $N = U_{2k}(q)$ and $X \cap N =N_1 =  U_{2k-1}(q)$ is factorized, appear for $N_{1}= U_3(3)$, $U_3(5)$, $U_3(8)$ or $U_9(2)$, by \cite[Table 3]{LPS}. The case $N = U_{4}(3)$  is excluded by 1(8) as mentioned above. Hence we should study the cases $N = U_{4}(5)$, $U_4(8)$ and $U_{10}(2)$.

For  $N = U_4(5)$ we have $|X \cap N |=|U_3(5)|=2^4\cdot 3^2\cdot 5^3\cdot 7$. If $A \leq X$, then by order arguments, it should be $\pi=\{7\}$, a contradiction.  By similar arguments, we get a contradiction for $N = U_3(8)$ when  $X \cap N=N_1 = U_3(8)$ and  $A \leq X$.

Suppose $N = U_{10}(2)$ and $A \leq X$, so that $A \cap N \leq N_1 = U_9(2)$. In this case, there exists an element $a \in A \cap N$ of order $r=19$ whose centralizer in $G$ is an abelian $2'$-group (recall that a field automorphism does not centralize any element of order $r$).  By 1(9), $A$ should be a $2'$-group. This leads to a contradiction by order arguments in all cases, except possibly when $Y \cap N=P_{5}$ is a parabolic subgroup. In this last case, $|G:Y|=\prod_{i=2}^{5}(q^{2i-1}+1)(q+1)$ for $q=2$ (see  \cite[3.3.3]{LPS}), and this number divides $|A|$. Then we can consider the prime $s=q_{10}=11 \in \pi(A) $ and an element of order $s$ in $A \cap N$ whose centralizer in $N$ is an abelian group. Since $s \in \pi(B) \subseteq \pi'$, this means that $A_{\pi}$ should be abelian, which contradicts 1(5).

\smallskip

Therefore we may assume in all cases that $A \leq Y$, $B \leq X$. It follows that  $|G:X|$ divides $|A|$.  By \cite[3.3.3]{LPS}, it holds that

\smallskip
\hspace{1.2cm}$|N:N \cap X|=|G:X|=q^{m-1}(q^{m}-1)/(q+1).$

\smallskip
\noindent Then there exists a maximal torus $T_1$ of order $(q^{m}-1)/((m, q+1)(q+1))$ in $N$ and an element $a\in A \cap N$ of order $r=q_{m}$, such that $C_{N}(a) \leq T_1$. Since $T_1$ is an abelian $p'$-group, it follows by 1(9)  that $A \cap N$ is a $p'$-group and $|A|_p \leq q$. But this is a contradiction, since $q^{m-1}$ divides $|A|$.

\end{proof}

\begin{lem}\label{Psp}  $N$ is not isomorphic to $PSp_{2m}(q)$, neither to $P\Omega_{2m+1}(q)$, $m \geq 2$.
\end{lem}
\begin{proof}
By comparing \cite[Theorem 8.8~(Table 7)]{VRsurvey} (see also \cite[Theorem 4.3]{Gro1}) with \cite[Theorem 6.9. Condition V]{VRsurvey}, we can deduce that if $N \in \{ PSp_{2m}(q), P\Omega_{2m+1}(q) \}$   is an $E_{\pi}$-group for a set of primes $\pi$ with $2, p \not\in \pi$, then it is also a $D_{\pi}$-group, so we get a contradiction.
\end{proof}

\begin{lem}\label{Omegaminus} $N$ is not isomorphic to $P\Omega_{2m}^-(q)$, $m\geq 4$, and $m\geq 5$ for $q=2$.

\end{lem}
\begin{proof}  Recall that $P\Omega_{8}^-(2) \cong L_2(q^2)$, so this case has been analyzed in Lemma~\ref{linear}. We consider first the factorizations appearing in  Table 1  in \cite{LPS} (with the same notation).  Here $q_{2m}$ always exists.

\medskip

\noindent \textbf{Case $X \cap N=\hat{}GU_m(q)$}, $m$ odd. In such case,  $Y \cap N=P_1$, or  $Y \cap N=N_1$.

Suppose first that $A\leq X$, so that $\pi(| G:Y|) \subseteq \pi(A)$ and $r=q_{2m} \in \pi(A)$. If $a$ is an element of order $r$ in $N \cap A$, then its centralizer in $X \cap N$ is contained in a torus of order dividing $(q^m+1)$, and so $r \in \pi$, $A$ is soluble and $A \cap N$ is a $p'$-group,  by 1(9).  If $Y \cap N=N_1$, then $|G:Y|=(1/(2, q-1))q^{m-1}(q^m+1)$ by \cite[3.5]{LPS} and this gives a contradiction by order arguments.
If $Y \cap N=P_1$, then $(q^{m-1}-1)(q^m+1)$ divides $|G:Y|$, and so $|A|$. In  particular, $t=q_{m-1} \in \pi(A \cap N)$ (this prime exists if $(m, q) \neq (7,2)$, since $m \geq 4$). Then by 1(7) $N$ would have a soluble maximal subgroup of order divisible by $t$ and $r$,  so they should divide  $2m(q^m+1)$ by \cite[Lemma 2.8]{ACK}. This can only happen if $m=t$ prime and $q \in \{2, 3, 5\}$ (by \cite[2.4, Proposition D]{LPS}), but this contradicts the fact that $A \cap N$ is soluble and its order also divides $2m(q^m+1)$, which is not the case. In the case when
$(m, q) = (7,2)$ we get also a contradiction because $G \leq N.2$ and the order of a maximal soluble subgroup of $N$ divisible by $r=51$ is not divisible by $q^{m-1}-1=7\cdot 9$.

Hence, we may assume that $A\leq Y$. This means that $s=q_{2m-2}\in \pi(A)$ (recall that $m$ is odd). In this case there is a torus $T \leq  N$ of order $(q^{m-1}+1)(q-1)/(4, q^m+1)$ containing the centralizer in $N$ of an element $a\in A \cap N$ of order $s$. This is again a contradiction, by 1(9).

\medskip
\noindent \textbf{Case $X \cap N=\Omega_m^{-}(q^2).2$}, $Y \cap N=N_1$.
Here $G=\Aut(N)=XY$ and $q\in\{2,4\}$ (see \cite[3.5.1]{LPS}). Note that $|G:Y|=q^{m-1}(q^m+1)$.

If $q=2$, then $G=O_{2m}^-(2)=N.2$ and $X \cap N=\Omega_m^-(4).2$. Suppose first that $A\leq X$, $Y\leq B$. Since $|G:Y|$ divides $|A|$, it follows that $r=q_{2m}\in\pi(A)$. If  $a \in A \cap N$ is an element of order $r$, then $C_N(a)$ is contained in a torus of order $(q^m+1)$. This provides a contradiction by 1(9).

Now we may assume that $A\leq Y$, so that $s=q_{2m-2}\in \pi(A)$ (recall that $m \geq 5$ in this case). Let $a$ be an element of order $s$ in $A \cap N$. It follows that $C_N(a)$ is contained in a torus of order $(q^{m-1}+1)(q+1)/(4, q^m+1)$. This leads again to a contradiction by 1(9), because $|N \cap B|_2|G/N|_2 < |N|_2$.

If $q=4$, then $G=N.4$ and $X \cap N=\Omega_m^-(16).2$.  Suppose first that $A\leq X$, and so $|G:Y|$ divides $|A|$. Then $r=q_{2m}\in\pi(A)$. But if $a$ is an element of order $r$ in $N\cap A$, then $C_N(a)$ is contained in a torus of order $(q^m+1)$, a contradiction by 1(9).

Now we may assume that $A\leq Y$. So that $s=q_{2m-2}\in \pi(A)$. Let $a$ be an element of order $s$ in $A \cap N$. It follows that $C_N(a)$ is contained in a torus of order $(q^{m-1}+1)(q+1)$. Then $s \in \pi$ and so $A \cap N$ is of odd order,  by 1(9), which is not the case.

\medskip

\noindent \textbf{Case }$X \cap N=GU_m(4), Y \cap N=N_2^+$. Here $G=\Aut(N)=N.4$,  and $m$ is odd.

In this case,  by \cite[3.5.2(c)]{LPS}, for $q=4$, it holds:
\smallskip

\hspace{2.5cm}$|G:Y|=(1/2)\frac{q^{2m-2}(q^m+1)(q^{m-1}-1)}{(q-1)}.$

\smallskip
If $A\leq X$, then $r=q_{2m} \in \pi(A)$. Then there exists an element $a \in A \cap N$ of order $r$ such that $C_{N}(A)$ is contained in a torus of order $q^{m}+1$. Then,  by 1(9), $A \cap N$ is of odd order and $|A|_2 \leq 4$, which is a contradiction since $|G:Y|$ divides $|A|$.

Therefore $A\leq Y$ and  $s=q_{2m-2}\in \pi(A)$. Since $N$ contains a maximal torus of order $(q^{m-1}+1)(q-1)$, which is the centralizer of an element $a\in A\cap N$ of order $s$, this means that $s\in \pi$ and $A \cap N$ is of odd order. In fact, there is no field automorphism centralizing an element of order $s$, hence $C_{G}(a)$, and so $A$, is of odd order. But this is again a contradiction, because $|G/N|=4$ in this case.

\medskip

It remains to consider the factorization for the case $N=P\Omega^{-}_{10}(2)$ which appears in \cite[Table 3]{LPS}. In this case $X \cap N=A_{12}$ and $Y \cap N=P_1$. Since the alternating group $A_{12}$ does not contain Hall $\pi$-subgroups with $ 2 \not\in \pi$ and $|\pi| \geq 2$, it should be  $A \leq Y$, and $B \leq X$. In this case $17 \in \pi(A)$, and  $G$ has a self-centralizing Sylow $17$-subgroup, which contradicts 1(6).

The lemma is proved.
\end{proof}

\begin{lem}\label{Omegaplus}  $N$ is not isomorphic to $P\Omega_{2m}^+(q), m\geq 4$.
\end{lem}

\begin{proof} Assume first that  $N = P\Omega_{2m}^+(q)$, for $ m>4$, and consider the factorizations appearing in \cite[Table 1]{LPS}.

\medskip

\noindent \textbf{Case }$X\cap N=N_1$. Here $|G:X|=\frac{1}{(2,q-1)}q^{m-1}(q^m-1).$ Let $d=(4, q^m-1)$.

Suppose first  that $A\leq X$ and $B\leq Y$. We distinguish the different possibilities for $Y_{0} \unlhd Y \cap N$.

$Y \cap N=P_m$ or $Y \cap N=P_{m-1}$ (stabilizers of totally singular $m$-subspaces). Then $|G:Y|=\prod_{i=1}^{m-1}(q^i+1)$ (see  \cite[3.6.1]{LPS}). Moreover $|G:Y|$ divides $|A|$,  and $r=q_{2m-2}\in \pi(A)$ (this prime exists since $m>4$).  There exists an element $a \in A \cap N$ of order $r$ such that $C_{N}(A)$ is contained in a torus of order  $(1/d)(q^{m-1}+1)(q+1)$. By 1(9), $r \in \pi$ and   $A \cap N$ is a $p'$-group. If $(q,m)\ne (2,5)$, then there exists a prime $s=q_{2m-4}\in \pi(A)$. Then $N$ must have a soluble subgroup of order divisible by $r$ and $s$, which is not the case by  \cite[Lemma 2.8]{ACK}.  If $(q,m)=(2,5)$, then $|G:Y|=3^3\cdot 5\cdot 17$. But $N$ has no subgroups of order $5\cdot 17$, a contradiction.

 $ Y \cap N=\hat{}GL_m(q).2$.  Note that $G \geq N.2$, when $m$ is odd. Since $|G:Y|$ divides $|A|$, we have $r=q_{2m-2} \in \pi(A)$ again.  If $a \in A \cap N$ is an element  of order $r$, then $C_{N}(a)$  is contained in a torus of order $(1/d)(q^{m-1}+1)(q+1))$. Hence $|A|_p \leq q$, by 1(9). But this is not the case, since $|G:Y|_p >q $.

 $Y \cap N=\hat{}GU_m(q).2$, $m$ even. In this case $s=q_{m-1}\in \pi(A)$ and the centralizer in $N$ of an element $a \in A \cap N$ of order $s$ is  contained in a torus of order $(1/d)(q^{m-1}-1)(q-1)$. Hence, by 1(9), it holds $|A|_p \leq q$, which is not possible since $|G:Y|$ divides $|A|$.

 $Y \cap N =\Omega_m^+(4). 2^2, q=2$, $m=2k$ even. Then $r=q_{2m-2}\in \pi(A), s=q_{2m-4} \in \pi(A)$. Arguing as in the previous subcases, we deduce that $r \in \pi$. But $N$ does not contain any soluble subgroup whose order is divisible by $r$ and $s$ by  \cite[Lemma 2.8]{ACK}, so we obtain a contradiction with 1(7).

 $Y \cap N = \Omega_m^+(16). 2^2, q=4$, $m$ even. Then again $r=q_{2m-2}\in\pi(A)$, $s=q_{2m-4} \in \pi(A)$ and we obtain a contradiction as  above.

 $Y_{0}=PSp_2(q)\otimes PSp_m(q)$, where $m=2k$ even, $q>2$. Here $Y_{0}$ has index $1$ or $2$ in $Y \cap N$. In this case also $r=q_{2m-2}\in \pi(A)$. If  $a \in A \cap N$ is an element of order $r$, then $C_{N}(A)$ is contained in a torus $T_1$ of order $(1/d)(q^{m-1}+1)(q+1))$. By 1(9), $r \in \pi$ and $|A|_p \leq q$, which is not case, since $|G:Y|$ divides $|A|$.

\smallskip

Now we may assume  that  $A \leq Y$, $B\leq X$ (for the case $X \cap N=N_1$ under consideration). This implies that $|G: X|=(1/(2,q-1))q^{m-1}(q^m-1)$ divides $|A|$.

If $m$ is odd, we consider  $s=q_m \in \pi(A)$. In this case $C_N(a)$ for $a \in A \cap N$ of order $s$ is contained in a torus of order $(1/d)(q^m -1)$. By 1(9), it should be $|A|_p \leq q$, which is not the case.

We distinguish now the different possibilities for $Y_{0} \unlhd Y \cap N$ when $m$ is even.

$Y \cap N=P_m$ or $P_{m-1}$. In this case we have that $s=q_m \in \pi(A)$ and the centralizer in $Y \cap N$ of an element $a \in A \cap N$ of order $s$  is contained in an abelian subgroup of $Y \cap N$ of order $q^{m}-1$. Hence, by 1(9), $|A|_p \leq q$, which is a contradiction.

$Y \cap N=\hat{} GL_m(q).2$, $m$ even. We can argue as in the previous case, since  $s=q_m \in \pi(A)$ and the centralizer in $Y \cap N$ of an element $a \in A \cap N$ of order $s$  is again an abelian $p'$-group.

$Y \cap N=\hat{} GU_m(q).2$, $m=2k$ even. We can choose a primitive prime divisor $u=q_k \in\pi(A)$ such that the centralizer of $a \in A \cap N$ of order $u$ in $Y \cap N$ is contained in a torus of order $(q^k+1)/(q+1, m)(q+1)$, and argue as above.

 $Y_{0}=PSp_2(q)\otimes PSp_m(q)$, $m$ even, $q > 2$.   If $a \in A \cap N$ is an element of order $r=q_m$, then its centralizer in $Y \cap N$ is  contained in $T \times L_2(q)$, where $T$ is an abelian $p'$-group  (recall $PSp_2(q) \cong L_2(q)$). If $r \in \pi$, then $|A \cap N|_p \leq q$, which gives a contradiction since $q^{m-1}$ divides $|A|$ and $m > 4$. Now, if $r \in \pi'$, then $A_{\pi}$ is a Hall $\pi$-subgroup of $T \times L_2(q)$, so it should be abelian, a contradiction.

$Y \cap N =\Omega_m^+(q^2). 2^2, q \in\{2, 4\}$, $m=2k$. In theses cases we consider the prime $t=q_{k} \in \pi(A)$ and  a torus in $Y \cap N$ of order dividing $(q^k-1)$ which is an abelian $2'$-group. We get a contradiction by 1(9).
\medskip

\noindent \textbf{Case $ X \cap N= P_1$}. Here $Y \cap N=\hat{}GU_m(q).2$, $m$ even.

If $A \leq X$ and $B \leq Y$,  we argue as in the case when $X \cap N =N_1$ and $Y \cap N =\hat{}GU_m(q).2$, $m$ even.

If $A\leq Y, B\leq X$, then $|G:X|$ divides $|A|$. Therefore $r=q_{2m-2}$, $t=q_{m}\in \pi(A)$. The centralizer of an element $a \in A \cap N$ of order $t$ in $Y \cap N$ is contained in an abelian subgroup of order $q^{m}-1$. Hence, by 1(9), $t \in \pi$. But  there is no soluble subgroup in $N$ of order divisible by $r$ and $t$, by \cite[Lemma 2.8]{ACK}, so we obtain a contradiction with 1(7).

\medskip

\noindent \textbf{Case $X \cap N =N_2^{-}$}. Here $|G:X|=(1/2)q^{2m-2}(q^m-1)(q^{m-1}-1)/(q+1)$, and  $(X/Z(X))'\cong P\Omega^{-}_{2m-2}(q)$.

First assume $A \leq X$, $B \leq Y$.

Let $Y \cap N =\hat{}GL_m(q).2$, $q\in\{2,4\}$. For $q=2$, $G \geq N.2$ if $m$ is odd, and for $q=4$,  $G \geq N.2$ and $G \neq O_{2m}^{+}(4)$.  Since $|G:Y|$ divides $|A|$, it holds $r=q_{2m-2} \in \pi(A).$ The centralizer in $N$ of an element of order $r$ in $N\cap A$ is an abelian $2'$-group. Hence, by 1(9), it holds that $|A|_2 \leq q$, which is not the case by order arguments.

Assume now that $Y \cap N=P_m$ or $P_{m-1}$. Since $|G:Y|=\prod_{i=1}^{m-1}(q^i+1)$, it holds that  $ r=q_{2m-2}, s=q_{2m-4} \in\pi(A)$ ($s$ exists when $(q,m)\neq(2, 5)$). Again the centralizer in $N$ of an element of order $r$ in $N\cap A$ is an abelian $p'$-group, so $r \in \pi$ by 1(9). But then by 1(7) $N$ should have a soluble subgroup of order divisible by $r$ and $s$, which is not the case (see \cite[Lemma 2.8]{ACK}). The case  $(q,m)=(2, 5)$ is excluded as above in the Case $X\cap N=N_1$.

\smallskip
Now we may suppose $A \leq Y$, $B \leq X$.

Consider $s= q_{m-1}$ when $m$ is even, and $s=q_m$ when $m$ is odd. In any case, $s\in \pi(A)$, since $|G:X|$ divides $|A|$.

For  any choice of $Y \cap N$, we can find a torus $T$ in $N$, which is an abelian $p'$-group, and   $C_N(a)\leq T$ for an element $a$ of order  $s$ in $A \cap N$ (see Lemma~\ref{primitive}). By 1(9), $s \in \pi$ and $|A|_p \leq q$, a contradiction since $|G:X|$ divides $|A|$.

%Assume that $Y \cap N = \hat{} GL_m(2).2$. If $m\ne 6$, this will be excluded as in the previous subcase. When $m=6$, then $|G|=2^{30}(2^{10}-1)(2^8-1)(2^6-1)(2^4-1)(2^2-1)$. By the above, $s=q_5=31\in \pi(A)$ and $|A|_2 \leq 2$, which is not the case.
%
%The subcase $Y \cap N =  \hat{}    GL_m(4).2$  is excluded arguing as above.

\medskip

\noindent \textbf{Case $ X \cap N= N_2^+$}. In this case $Y \cap N=\hat{}GU_m(4), q=4, m$ even, and $G=N.2$. Here we have $ |G:X|=\frac{1}{2}q^{2m-2}(q^{m}-1)(q^{m-1}+1)/(q-1)$, by \cite[(3.6.3 c)]{LPS}.

If $A\leq X$, we can take $q_{m-1}, q_{2m-4}\in \pi(A)$, and we can argue as in previous cases to get a contradiction with 1(7).

Assume now $A \leq Y$, so $|G:X|$ divides $|A|$. In this case,  $r=q_{2m-2}\in \pi(A)$ and we can follow previous arguments to get a contradiction with 1(9).

\medskip

Next we consider some extra factorisations appearing in \cite[Tables 2, 3]{LPS}.

Let $N = P\Omega_{16}^+(q)$, $X \cap N=\Omega_{9}(q).a$, $a \leq 2$, $Y \cap N=N_{1}=\Omega_{15}(q)$. If $A \leq X$,  we consider $r=q_{14} \in \pi(A)$. If $A \leq Y$ we consider $s=q_{8} \in \pi(A)$ and a torus in $Y \cap N$ containing the centralizer of an element of order $s$ in $A \cap N$. The contradiction arises in both  cases applying 1(9), because in any case $|A|_p > q$.

\smallskip

 Let $N=\Omega_{24}^+(2), X\cap N=N_1=\Omega_{23}(2), Y\cap N=Co_1$. If $A\leq X$, then we consider $r=q_{22} \in \pi (A)$, to get a contradiction as in previous cases by 1(9). If $A \leq Y$, since the only Hall $\pi$-subgroups of $Co_1$ with $|\pi| \geq 2$ appear for $\pi=\{11, 23\}$ and a Sylow $23$-subgroup in this group is self-centralizing, we get a contradiction by 1(6).

\medskip
Consider now the case $m=4$, i.e. $N =P\Omega_8^+(q)$ (see Table 4 in \cite{LPS}).

By 1(8), we can assume that $q \neq 2$, since $|\pi(P\Omega_8^+(2))|=4$.
First note that, by order arguments and by checking all possibilities for $X\cap N$ and $Y \cap N$ in \cite[Table 4]{LPS}, it holds that $q^4-1$ should divide $(|N \cap A|, |N \cap B|)$. This means in particular that $s=q_4 \in \pi(q^4-1) \subseteq \pi'$. (Note that when $q=3$, then $q_4=5$.)

In any of the cases for $X  \cap N$ and $Y \cap N$ containing $A \cap N$, except the one below, it holds that there exists  a torus of order divisible by $s$ and an element $a \in A \cap N$ of order $s$
whose centralizer in $A \cap N$ is contained such torus, which is an abelian $p'$-group. But this means by 1(9) that $s \in \pi$, which is a contradiction.

In the particular case $ A \cap N \leq Y \cap N =(PSp_2(q) \times PSp_4(q)).2$, since $s \in \pi'$, it holds that $A_{\pi}$ should be a Hall $\pi$-subgroup of  $C_{Y \cap N}(a)$, with $a\in A\cap N$ of order $s$, and so $A_\pi$ should be an abelian group, because $\pi(PSp_2(q)) \subseteq \pi'$, which is again a contradiction.

There is an exceptional factorization when $N=P\Omega_8^+(3)$, $X \cap N=2^6.A_8$, $Y \cap N \in \{P_1, P_3, P_4\}$, $G \geq N.2$. Since the alternating group does not have Hall $\pi$-subgroups for $|\pi| \geq 2$, it holds that $A \leq Y$. But then $q_4=5 \in \pi(A)$ and the previous argument holds also in this case, which concludes the proof of the lemma.
\end{proof}

The Main Theorem  is now  proved.

\bigskip

\noindent
{\bf Acknowledgments.} Research supported by Proyectos PROMETEO/2017/ 057 from the Generalitat Valenciana (Valencian Community, Spain), and PGC2018-096872-B-I00 from the Ministerio de Ciencia, Innovaci\'on y Universidades, Spain, and FEDER, European
Union; and first author also by Project VIP-008 of Yaroslavl P. Demidov State University.

\bigskip

\noindent
 \footnotesize{L. S. KAZARIN}\\
\footnotesize{Department of Mathematics, Yaroslavl P. Demidov State University}\\
 \footnotesize{Sovetskaya Str 14, 150014 Yaroslavl, Russia}\\
 \footnotesize{E-mail: Kazarin@uniyar.ac.ru}\\
 \\
 \footnotesize{A. MART\'{I}NEZ-PASTOR}\\
 \footnotesize{Instituto Universitario de Matem\'{a}tica Pura y  Aplicada IUMPA}\\
\footnotesize{Universitat Polit\`{e}cnica de Val\`{e}ncia,
 Camino de Vera, s/n,  46022 Valencia, Spain}\\
 \footnotesize{E-mail: anamarti@mat.upv.es}\\
 \\
 \footnotesize{and M.~D. P\'{E}REZ-RAMOS}\\
\footnotesize{Departament de Matem\`{a}tiques, Universitat de Val\`{e}ncia,}\\
 \footnotesize{C/ Doctor Moliner 50, 46100 Burjassot
({Val\`{e}ncia}), Spain}\\
\footnotesize{E-mail: Dolores.Perez@uv.es}


\begin{thebibliography}{AAA}

\bibitem{AFG} B. Amberg, S. Franciosi, and  F. de Giovanni, {\it Products of Groups},
(Clarendon Press, Oxford, 1992).
%\bibitem{AKA} B. Amberg and L. S. Kazarin, On finite products of soluble groups, {\it Israel J. Math.} {\bf 106} (1998), 93-108.

\bibitem{ACK} B. Amberg, A. Carocca, and L.S. Kazarin,
Criteria for the solubility and non-simplicity of finite groups, {\it J. Algebra} {\bf 28}5 (2005),  58--72.
\bibitem{A-C} Z. Arad and D. Chillag, Finite groups containing a nilpotent Hall subgroup of even order,  {\it Houston J. Math.}   {\bf 7}  (1981), 23--32.
\bibitem{A-F} Z. Arad and E. Fisman, On finite factorizable groups,  {\it J. Algebra} {\bf86} (1984), 522--548.
 \bibitem{Ber}  Y.~G. Berkovich, Generalization of the theorems of Carter and Wielandt, {\it Sov. Math. Dokl.}  {\bf 7} (1966), 1525--1529.
\bibitem{Atlas} J.~H .Conway, R.~T. Curtis, S.~P. Norton, R.~A. Parker and R~.A. Wilson,  {\it Atlas of finite groups: maximal subgroups and ordinary characters of simple groups},  (Clarendon Press, Oxford, 1985).
\bibitem{D-H} K. Doerk  and T. Hawkes, \emph{Finite soluble groups}, (Walter De Gruyter, Berlin-New York, 1992).
\bibitem{Gro1} F. Gross,  Odd order Hall subgroups of $GL_n(q)$ and $Sp_{2n}(q)$, {\it Math. Z.} {\bf 187} (1984), 185--194.
\bibitem{Gro2} F. Gross,  Conjugacy of odd order Hall subgroups, {\it Bull. London Math. Soc.} {\bf 19} (1987), 311--319.
\bibitem{Hall}  P. Hall, Theorems like Sylow's,  {\it Proc. London Math. Soc.} \textbf{3}(2) (1956), 286--304.
%\bibitem{Hup} B. Huppert, {\it Endliche Gruppen I}, (Springer-Verlag, Berlin, 1967).
\bibitem{Ka1}  L.~S.  Kazarin, Criteria for the nonsimplicity of factorable groups,  {\it Izv. Akad. Nauk SSSR}, Ser. Mat. {\bf 44} (1980), 288--308.
\bibitem{Ka2} L.~S. Kazarin, On groups which are the product of two soluble groups, {\it Comm. Algebra} {\bf 14} (1986), 1001--1066.
\bibitem{Ka3} L.~S. Kazarin, Factorizations of finite groups by solvable subgroups,  {\it  Ukr. Mat. J.} {\bf 43}(7) (1991), 883--886.
\bibitem{paper1}  L.~S. Kazarin, A. Mart\'{\i}nez-Pastor, and M.~D. P\'{e}rez-Ramos, On the product of a $\pi$-group and a $\pi$-decomposable group, {\it J. Algebra} {\bf 315} (2007), 640--653.
\bibitem{paper2} L.~S. Kazarin, A. Mart\'{\i}nez-Pastor, and M.~D. P\'{e}rez-Ramos,  On the product of two  $\pi$-decomposable soluble groups, {\it Publ. Mat.} {\bf 53} (2009), 439--456.
\bibitem{bath} L.~S. Kazarin, A. Mart\'{\i}nez-Pastor, and M.~D. P\'{e}rez-Ramos, Extending the Kegel-Wielandt theorem through $\pi$-decomposable groups, in: {\it Groups St Andrews 2009 in Bath}, vol. 2,  Lond. Math. Soc. Lecture Note Ser. {\bf 388}, (Cambridge University Press, Cambridge, 2011), pp. 415--423.
\bibitem{paper3} L.~S. Kazarin, A. Mart\'{\i}nez-Pastor, and M.~D. P\'{e}rez-Ramos, A reduction theorem for a conjecture on products of two  $\pi$-decomposable  groups, {\it J. Algebra} {\bf 379} (2013), 301--313.
\bibitem{paper4} L.~S. Kazarin, A. Mart\'{\i}nez-Pastor, and M.~D. P\'{e}rez-Ramos, On the product of two  $\pi$-decomposable  groups, {\it Rev. Mat. Iberoam.} {\bf 31} (2015), 51--68.
\bibitem{trif} L.~S. Kazarin, A. Mart\'{\i}nez-Pastor, and M.~D. P\'{e}rez-Ramos, Finite trifactorized groups and $\pi$-decomposability, Bull. Aust. Math. Soc. \textbf{97} (2018), 218--228.
\bibitem{LiX}C.~H. Li, B. Xia, Factorizations of almost simple groups with a solvable factor, and Cayley graphs of solvable groups, preprint. Available online at https://arxiv .org /abs /1408 .0350.
\bibitem{LPS} M.Liebeck, C.~E. Praeger, and J. Saxl, The maximal factorizations of the finite simple groups and their automorphism groups,  {\it Mem. Amer. Math. Soc.} {\bf 86}, No. 432, (Amer. Math. Soc., Providence, RI, 1990).
\bibitem{RV} D.~O. Revin and E.~P. Vdovin,  Hall subgroups in finite groups,  in: \emph{Ischia Group Theory 2004},  Contemp. Math. {\bf 402}, (Amer. Math. Soc., Providence, RI, 2006), pp. 229--263.
\bibitem{Row} P.~J. Rowley, The $\pi$-separability of certain factorizable groups, {\it Math. Z.} {\bf 153} (1977), 219--228.

%\bibitem{Stein} R. Steinberg, Lectures on Chevalley groups, Yale Univ., 1967.
\bibitem{VV} A.~V. Vasiliev, E.~P. Vdovin, An adjacency criterion for the prime graph of a finite simple group, \textit{Algebra and Logic} {\bf 44}(6) (2005), 381--406.
\bibitem{VV2}  A.~V. Vasiliev, E.~P. Vdovin, Cocliques of maximal size in the prime graph of a finite simple group,  \textit{Algebra and Logic} {\bf 50}(4), (2011) 291--322.
\bibitem{VRsurvey} E.~P. Vdovin and D.~O. Revin, Theorems of Sylow type,  {\it Russian Math. Surveys} {\bf 66}(5) (2011), 829--870.
\bibitem{Wie} H. Wielandt, Zum Satz von Sylow, {\it Math. Z.} {\bf 60} (1954), 407--408.
\bibitem{Zsi} K. Zsigmondy, Zur Theorie der Potenzreste, Monatsh.
Math. Phys. {\bf 3} (1892), 265-284.

\end{thebibliography}
\end{document}